\documentclass[11pt,twoside]{amsart}

\usepackage[english]{babel}
\usepackage{amsmath, amsthm, amssymb, amsfonts}
\usepackage{latexsym}
\usepackage{graphicx}
\usepackage{pdfpages}
\usepackage {bm}
\usepackage {indentfirst} 

\usepackage{hyperref,cite}

\advance\hoffset by -1.5cm

\newcommand{\T}{\mathbb{T}}
\newcommand{\D}{\mathbb{D}}
\newcommand{\C}{\mathbb{C}}
\newcommand{\N}{\mathbb{N}}

\newcommand{\B}{\mathcal{B(X)}}
\newcommand{\R}{\mathcal{R}}
\newcommand{\n}{\mathcal{N}}
\newcommand{\h}{\mathcal{H}}
\newcommand{\x}{\mathcal{X}}
\newcommand{\y}{\mathcal{Y}}

\newtheorem{example}{Example}[section]
\newtheorem{theorem}[example]{Theorem}
\newtheorem{lemma}[example]{Lemma}

\newtheorem{proposition}[example]{Proposition}
\newtheorem{corollary}[example]{Corollary}
\newtheorem{remark}[example]{Remark}

\numberwithin{equation}{section}

\begin{document}

\title[On ergodic operator means in Banach spaces]{{\bf On ergodic operator means in Banach spaces}}

\author[A. Aleman]{Alexandru Aleman}
\address{Lund University, Mathematics, Faculty of Sciences,
P.O. Box 118 S-221 00, Lund, Sweden}
\email{Alexandru.Aleman@math.lu.se}

\author[L. Suciu]{Laurian Suciu}
\address{Department of Mathematics, "Lucian Blaga" University
	of Sibiu, Dr. Ion Ra\c tiu 5-7, Sibiu, 550012, Romania}
\email{laurians2002@yahoo.com}

\date{}


\begin{abstract} We consider a large class of operator means and prove that a number of ergodic theorems, as well as growth estimates
known for particular cases, continue to hold in the general context under
 fairly mild regularity conditions. The methods developed in the paper not only yield a new approach based on a general point of view, but also lead to results that are new, even in the context of the classical Ces\`aro means.
\end{abstract}

\maketitle

\section{Introduction}
\medskip
Let $\x$ and $\y$ be two Banach spaces, and $\mathcal{B}(\x, \y)$ be the Banach space of all bounded linear operators from $\x$ into $\y$. Consider $\B=\mathcal{B}(\x,\x)$ as a Banach algebra and $I \in \B$ the identity operator on $\x$. The range and the kernel of $T\in
\B$ will be denoted by $\R(T)$ and $\n(T)$, respectively. Also,
$\sigma(T)$ and $\sigma_p(T)$ stand for the spectrum and the point
spectrum of $T$, and  we denote by $\x^*$ and $T^*$ the dual space of $\x$ and the adjoint operator of $T$.

The present paper is concerned with sequences $\{T_n\}$ in
 the closed convex hull of the powers of a fixed operator $T\in \B$ of the form
\begin{equation}\label{ec11}
T_n=\sum_{j\ge 0} t_{nj}T^j, \quad t_{nj}\ge 0\,,\quad \sum_{j\ge 0}t_{nj}=1,
\end{equation}
and if $t_{nj}>0$ for infinitely many values of $j$,  the above sum converges in norm.

We shall denote from now on by $\kappa(T)$ the set of such, possibly infinite, convex combinations of powers of $T$
and we refer to sequences in $\kappa(T)$. Actually, we formally apply an infinite matrix to the operator sequence $\{T^n\}$.

 We let
 \begin{equation}\label{ec12}
 \x_0=\{x\in \x:~\|T_nx\|\to 0, \hspace*{1mm} n \to \infty\},
 \end{equation}
 and
 \begin{equation}\label{ec13}
 \x_c=\{x\in \x:~\{T_nx\}\,\text{converges in norm} \}.
 \end{equation}

This general type of operator  mean has already  been considered in the literature, for example in [C], [Eb], [Kr], and more recently, in [R] and [Sc], but only under the additional assumption that $T$ is power bounded, that is $\{T^n\}$ is bounded in norm.

Let us notice that in [C] such sequences $\{T_n\}$ are defined with respect to a regular matrix $\{t_{nj}\}$ which implies, in particular, that $\lim_{n\to \infty} t_{nj} =0$ for $j=0,1,...$.

The basic idea behind the results presented in this paper is that a great deal of estimates, or ergodic theorems for specific operator means, continue to hold in a fairly general context than the power boundedness, and they can be deduced from either:

1)  A regularity condition for the sequence  $\{T_n \}$ (which does not require the regularity of the matrix $\{t_{nj}\}$),\\
or,

2) Conditions invariant to rotations.

\medskip
The regularity condition considered here is that for some fixed, nonnegative integer $n_0$, we have $TT_n-T_{n+n_0} \to0$ either strongly, or in the operator norm. It is used in order to exhibit useful isometries related to $\{T_n\}$ that are acting on quotient spaces of $X$. This is a method that appears also in the works of [K\'e] (concerning the generalized Banach limits), and is a powerful tool in the study of asymptotic behavior (see for instance [Al], [V]).

By a rotational invariant condition we mean a growth restriction obeyed not only by the sequence $\{T_n\}$, but also by all ''rotated'' sequences  $T_{n\lambda}$ obtained by applying the same convex combinations to the operator $\lambda T$. More precisely, for $T_n$ as in \eqref{ec11} we set
 \begin{equation}\label{ec14}
 T_{n\lambda}=\sum_{j\ge 0}t_{nj}\lambda^j T^j, \quad |\lambda |=1.
 \end{equation}
This type of condition brings complex analysis in the picture, a powerful tool in the study of bounded linear operators.

The paper is organized as follows. Section 2 serves to an introductory purpose. We present some general technical tools, including  properties of $\x_0$,  as well as a list of examples of sequences in $\kappa(T)$, satisfying the regularity conditions. Moreover, we introduce and discuss the {\it backward iterate} sequence $\{T_n^{(-1)}\}$ of a sequence $\{T_n\}$ in $\kappa(T)$, which is inspired by the  well known relations between Ces\'aro means of order $p\ge 1$ and the means of order $p+1$. The backward iterate is determined by  the identity   $$T_n^{(-1)}(T-I)=\left(\sum_{j\ge 1}jt_{nj}\right)^{-1}(T_n-I)\,,$$
and plays an important role in our considerations.

In Section 3 we investigate sequences in $\kappa(T)$ for supercyclic operators $T$ and we refer to the Ansari-Bourdon theorem [AB], which asserts that  for a power bounded and supercyclic operator $T$, the powers must converge strongly to zero. The result may fail for general sequences in $\kappa(T)$, because of the presence of eigenvalues on the unit circle of the adjoint $T^*$. However,  we are able to give a fairly complete picture for the subspaces $\x_0$ and $\x_c$ defined above (Theorem \ref{supercyclic}), under the regularity condition. Our general method based on the intertwining of $T$ with an isometry gives even in the power boundedness context, an improvement of Ansari-Bourdon's result, namely the fact that if $T$ is supercyclic and power bounded, then every sequence in $\kappa(T)$ which satisfies the regularity condition strongly converges to zero. Concrete examples of supercyclic operators which actually are Ces\'aro bounded but not power bounded, are deferred to Section 5. In particular, this answers a question raised by J. Bonet in a recent public communication concerning the existence of a Ces\`aro bounded hypercyclic operator.

Section 4 is devoted to conditions invariant to rotations. Our first main theorem (Theorem \ref{nevanlinna}) is a far-reaching generalization of a result of O. Nevanlinna [N]. More precisely, for operators $T\in \B$ with peripheral spectrum $\sigma(T)\cap\T$ of arclength measure zero we derive an estimate for the powers $T^n$, under the assumption that for some sequence $\{T_n\}$ in $\kappa(T)$, both sequences $\{T_n\}$ and $\{T_n^{(-1)}\}$ satisfy uniform rotational invariant growth restrictions. To avoid technicalities we shall only mention some applications of this theorem. For example, when applied to (discretized) Abel means the result yields the following Nevanlinna-type theorem: If $u:(1,\infty)\to (0,\infty)$ is decreasing, then such a behavior of the resolvent
$$\|(T-\lambda I)^{-1}\|\le \frac{u(|\lambda|)}{|\lambda|-1}\,,\quad |\lambda|>1\,,$$
implies $$ \|T^n\|=o\left(nu\left(\frac{n}{n-1}\right)\right), \quad n\to \infty\,.$$

The result applies to other means as well, in particular, to other functions of $T$. We pay special attention to the Ces\`aro means.  We prove that certain rotational invariant growth conditions for Ces\`aro means of any order $p\ge 2$ are actually equivalent to a Kreiss-type condition (Theorem \ref{kreisstype}). When $p=2$, a special case of this result was proved by J. Strikwerda and B. Wade in [SW]. Other similar considerations on the Kreiss resolvent condition were obtained by E. Berkson and T. A. Gillespie [BG] for trigonometrically well-bounded operators, and by O. El-Fallah and T. Ransford [FR] for operators satisfying a generalized Kreiss-type condition with respect to a compact subset of the unit circle.

 In the case $p=1$, our  conditions are equivalent to uniform estimates for the partial sums of the Taylor expansion at infinity of the resolvent function of $T$ (Theorem \ref{kreisstype}). This is related to the uniform Kreiss boundedness condition, a concept introduced recently in [MSZ]. For such conditions a weaker conclusion (in the sense of strong convergence) holds under an additional technical assumption (Theorem \ref{uniformKreiss}), and this can be applied to the supercyclic operators, for instance.

 Section 5 concerns examples which require more work. In addition to the construction of a Ces\`aro bounded hypercyclic operators already mentioned above, we have even a Ces\`aro bounded $3$-isometry (in the sense of [AS]). This operator has the remarkable property that for any $x\neq 0$, the Ces\`aro means of $T$ converge weakly, but not strongly, to zero. This is a somewhat surprising fact which is also  related to a well-known mean ergodic theorem (see [Kr]). It shows that even if the regularity condition holds, weak convergence of the means fails to imply their  strong convergence in a fairly dramatic way.

Finally, by modifying an example of A. L. Shields [Sh] we show that the conclusion of the generalized O. Nevanlinna theorem (Theorem \ref{nevanlinna}) may fail if the peripheral spectrum of $T$ has positive arclength measure on the unit circle.

Let us notice that at this moment we know only an example, which appear also in [Sh], of a uniformly Kreiss bounded operator on a Hilbert space, which is not power bounded. The operator in Subsection 5.1 is another such example which is also hypercyclic.

\section{Preliminaries and terminology}
\medskip

\subsection{General considerations}

Given $T\in \B$, we begin with two  observations about sequences $\{T_n\}$ in $\kappa(T)$.
\begin{proposition}\label{firstobs} Assume that the sequence $\{T_n\}$ in $\kappa(T)$ is bounded on
$\overline{\R(T-I)^m}$, for some nonnegative integer $m$. Then $\x_0$ is closed and
$$\x_0\subset\bigcap_{k\ge 0}\overline{\R(T-I)^k}\,,$$
 \end{proposition}
\begin{proof}
Clearly, the closure $\y$ of  $\x_0$ is invariant for $T,T_n$  and if $S$ denotes the restriction of $T$ to $\y$, it is easy to see that $S-I$ has dense range. Indeed, if $f\in \R(S-I)^\perp$, then
$f(Tx)=f(x)$ for all $x\in \y$, so that from the $T$-invariance and the fact that $T_n\in \kappa(T)$, we deduce that $$f(T_nx)=f(x)\,,$$
for all $x\in \y$ and all $n\in \N$. But then $f$ vanishes on the dense subspace $\x_0$, i.e. $f=0$. Since $S$ is bounded on $\y$ it follows that $(S-I)^k$ has dense range for all $k\in\N$, and we obtain
$$\x_0\subset\overline{(S-I)^k\y}\subset \overline{\R(T-I)^k}\,.$$
In particular, this gives  $\x_0\subset \overline{\R(T-I)^m}\,.$ and by the  uniform boundedness principle we conclude that $\x_0=\y$.
\end{proof}
A direct consequence of this result is that if $\x_0= \overline{\R(T-I)^k}$ for some nonnegative integer $k$ then
\begin{equation}\label{ascent-conseq} \overline{\R(T-I)^p}= \overline{\R(T-I)^k}\,,\quad  p\ge k\,.  \end{equation}

A simple additional condition that determines $\x_0$ is given below.

\begin{corollary}\label{thm23}  Assume that  there exist nonnegative integers $k,m$ with $k\ge m$, such that the sequence $\{T_n\}$ in $\kappa(T)$ is bounded on
$\overline{\R(T-I)^m}$, and that $\{T_n(T-I)^k\}$ converges strongly to zero. Then
$$\x_0= \overline{\R(T-I)^k}\,.$$
\end{corollary}

Our second observation  provides a tool to estimate  $\x_0$, and it applies to more general sequences $\{T_n\}$.  We shall use the following notations.

 Given a continuous seminorm $ \gamma$ on $\x$ we denote its kernel by $\n(\gamma)$, and consider the quotient space  $\x_{/\n(\gamma)} $ as a normed space with
$$
\|[x]\| =\gamma(x), \quad [x]=x+\n(\gamma)\,.
$$
We denote by  $\x_\gamma$  the completion of this space, and by $Q_\gamma$  the quotient map from $\x$ into $\x_\gamma$. Obviously,  $Q_\gamma\in
\mathcal{B}(\x,\x_\gamma)$ with $$\n(Q_\gamma)=\n(\gamma)\,,\quad \overline{\R(Q_\gamma)}=\x_\gamma\,.$$

\begin{proposition}\label{secondobs} Assume that $\{T_n\}\subset \B$  is bounded on
$\R(T-I)^m$, for some nonnegative integer $m$, and that each $T_n$ commutes with $T\in \B$. Let
\begin{equation}\label{gamma}
\gamma(x)=\limsup_{n\to \infty} \|T_n(T-I)^mx\|, \quad x\in \x\,.
\end{equation}
Then there exists a unique operator $V_\gamma \in \mathcal{B}(\x_\gamma)$, such that $Q_\gamma T=V_\gamma Q_\gamma$. Moreover,  $\sigma(V_{\gamma})\subset\sigma(T)$ and if $F$ is analytic in a neighborhood of $\sigma(T)$ and satisfies $F(V_\gamma)=0$, then $\x_0\subset
\overline{\R( F(T)(T-I)^m})$.
\end{proposition}\label{vgamma}

\begin{proof} The statement is almost self-explanatory. The equality $Q_\gamma T=V_\gamma Q_\gamma$ defines $V_{\gamma}$ on a dense subspace, and since $\n(Q_\gamma)=\n(\gamma)$ is invariant for $T$, it follows that $V_\gamma$ is bounded on this subspace, hence it has a unique bounded extension to $\x_\gamma$. In a similar way we obtain bounded linear operators $V_\lambda,~\lambda \in \C\setminus\sigma(T)$ with $Q_\gamma (T-\lambda I)^{-1}=V_\lambda Q_\gamma$, which implies that $V_\lambda=(V_\gamma-\lambda I)^{-1}$. Finally, for functions $F$ as in the statement we have  $Q_\gamma F(T)=F(V_\gamma)Q_\gamma$, hence  $Q_\gamma F(T)=0$ whenever $F(V_\gamma)=0$. This also gives the required inclusion, which finishes the proof.
\end{proof}
Due to the last assertion above, we are of course interested in situations where $V_\gamma$ is as simple as possible, in particular, when it has a large functional calculus algebra.\\

An immediate application for sequences in $\kappa(T)$ is given below.

\begin{corollary}
Assume that the sequence $\{T_n\}$ in $\kappa(T)$ is bounded on
$\overline{\R(T-I)^m}$, for some nonnegative integer $m$, and let $\gamma
$ be the seminorm given by \eqref{gamma}. Then $\x_0= \overline{\R(T-I)
^k}$ for some  $k\ge m$ if and only if $(V_\gamma-I)^{k-m}=0$.
\end{corollary}

Proposition \ref{secondobs} allows us to  extend to the ergodic context  the method developed in [Al], [K\'e], [V] and [SZ] in the study of  asymptotic properties of scalar weighted powers of an operator in $\B$.

\subsection{A regularity assumption} Given $T\in\B$ we shall consider  sequences $\{T_n\}$ in $\B$ which are bounded on $\overline{\R(T-I)^m}$, and satisfy the condition
\begin{equation}\label{reg}
\lim_{n\to\infty}\|(TT_n-T_{n+n_0})x\|=0\,,
\quad x\in \R(T-I)^m\,.
\end{equation}
for some non-negative integers $m,n_0$.

 A sequence $\{T_n\}$ satisfying the condition \eqref{reg} will be shortly called {\it $(n_0,m)$-regular}. In particular, we say that $\{T_n\}$ is {\it regular} when $n_0=1$ and $m=0$, and $\{T_n\}$ is {\it ergodic} (according to [Kr]) if it verifies \eqref{reg} with $n_0=m=0$. For different values of $m, n_0$,  these conditions of regularity are in general not related. In particular,  ergodicity cannot be obtained from the regularity, as we shall see
in the following examples  (including  the one constructed in Subsection 5.2).

Some direct consequences of the regularity condition are listed below.

\begin{theorem}\label{regobs}
Assume that $\{T_n\}\subset \B$ is $(n_0,m)$-regular. Then

(i) If $\{T_n\} \subset \kappa(T)$ is bounded on
$\overline{\R(T-I)^m}$, then $\x_0= \overline{\R(T-I)^k}$ for some  $k> m$ if and only if
$$\lim_{n\to\infty}\sum_{l=0}^{k-m}(-1)^l\binom{k-m}{l}T_{n+ln_0}x=0\,,\quad x\in\overline{\R(T-I)^m}\,.$$
Moreover, $\x_c\cap\overline{\R(T-I)^m}$ is the direct topological sum
$$\x_c\cap \overline{\R(T-I)^m}=\x_0\oplus (\n(T-I)\cap\overline{\R(T-I)^m})\,.$$

(ii) If each $T_n$ commutes with $T$, $\{T_n\}$ is bounded on
$\R(T-I)^m$ and $\gamma$ is the seminorm given by \eqref{gamma},
then the operator $V_\gamma$ from Proposition \ref{secondobs} is an isometry.
\end{theorem}

\begin{proof} (i) A repeated application of \eqref{reg} gives for $x\in \overline{\R(T-I)^m}$
$$(T-I)^{k-m} T_nx-\sum_{l=0}^{k-m}(-1)^l\binom{k-m}{l}T_{n+ln_0}x\to 0\,,\hspace*{1mm} n\to \infty \,,$$
which proves the first part. Also from \eqref{reg} we see that if $x\in\x_c$ then $$\lim_{n\to\infty}T_nx\in \n(T-I)\,,$$
which together with Proposition \ref{firstobs} proves that  $\x_c\cap\overline{\R(T-I)^m}$ is the algebraic  direct sum
of $\x_0$ and $\n(T-I)\cap\overline{\R(T-I)^m}$. If $x\in \x_0,~ y \in \n(T-I)\cap\overline{\R(T-I)^m}$ then we can choose $n$ such that $$\|T_n x\|<\|x-y\|\,.$$

Since $T_ny=y$ we obtain
$$\|y\|\le \|T_nx\|+\|T_n(y-x)\|< \|y-x\|(1+\sup_n\|T_n|\overline{\R(T-I)^m}\|)\,,$$
and (i) follows.

To see (ii) just apply again  \eqref{reg} to conclude that $\gamma(Tx)=
\gamma(x)$, and this assures that $V_{\gamma}$ is an isometry.
\end{proof}

One motivation for the condition \eqref{reg} is a result in [Ku] which asserts that for $m=0,~n_0=0,1$, this condition  is necessary
for the strong convergence of $\{T_n\}$ to the {\it ergodic projection} $P_T$ of $T$, that is $P_T \in \B$ with $\n(T-I)=\R(P_T)$ and $\overline{\R(T-I)}=\n(P_T)$.  More precisely,  the result in  [Ku] states that $T_n \to P_T$ strongly if and only if $\{T_n\}$ satisfies \eqref{reg} with $m=0$, $n_0=0,1$ and
\begin{equation}\label{almostconvergence}
\lim_{k\to \infty} \sup_{n\in \N} \|\frac{1}{k+1}\sum_{j=0}^k T_{n+j} x -P_Tx\|=0, \quad x\in \x.
\end{equation}
The last condition  is called  almost convergence and was introduced by G. G. Lorentz in [L].  One can obtain similar characterizations
 of the strong convergence of $\{T_n\}$  for $m=0$ and $n_0>1$, by using \eqref{reg} and an appropriate almost convergence of step $n_0>1$.

\subsection{Some examples}
A  large variety of examples of sequences satisfying \eqref{reg} arise from matrix summability methods (see [B]), more precisely, by applying such matrices to the sequence $\{T^n\}$.
The simplest choice of such matrices is the identity which gives the sequence $\{T^n\}$. It satisfies \eqref{reg} with $n_0=1$. The so-called Zweier matrix (see [B]) leads to the sequence $\{T_n\}$ with
\begin{equation}\label{zweier} T_n=\frac1{2}(T^{n-1}+T^n), \quad n\ge 1 \end{equation}
which also satisfies \eqref{reg} with $n_0=1$, that is $\{T_n\}$ is regular. Remark that $\{T_n\}$ is ergodic if and only if $\{T^n\}$ converges
strongly on $\R(T^2-I)$.

The most common examples considered in this area are
 the {\it Ces\`aro means of order} $p\in \N$ of an
operator $T\in \B$, which are the operators $M_n^{(p)}(T)$defined for $n\in
\N$ by $ : M_0^{(p)}(T)=I$,
$M_n^{(0)}(T)=T^n$, while for $n,p\ge 1$ as
\begin{align}\label{0} M_n^{(p)}(T) &= \frac{p}{(n+1)...(n+p)}\sum_{j=0}^n\frac{(j+p-1)!}{j!}M_j^{(p-1)}(T)\\&\nonumber
=\frac{p}{(n+1)...(n+p)}\sum_{j=0}^n \frac{(n-j+p-1)!}{(n-j)!}T^j\\&\nonumber
=\frac{p}{n+p}
\sum_{j=0}^n \prod_{k=1}^{p-1} (1-\frac{j}{n+k})T^j
\end{align}
In the case $p=1$ we write
$$
M_n(T):=M_n^{(1)}(T)=\frac{1}{n+1}\sum_{j=0}^n T^j
$$
and it is usually called the {\it Ces\`aro mean} of $T$.

If $\{M_n^{(p)}(T)\}$ is bounded or it strongly converges in $\B$, we say that $T$ is {\it $p$-Ces\`aro bounded (ergodic)}, respectively $T$ is {\it Ces\`aro bounded (ergodic)} when $p=1$.

These operator means have a considerably richer algebraic structure.
They satisfy  the following recurrent identities valid for
  $n,p\ge1$:
\begin{equation}\label{1}
 \qquad M_n^{(p)}(T)(T-I)= \frac{p}{n+1} (M_{n+1}^{(p-1)}(T)-I)\,,
\end{equation}
\begin{equation}\label{2}
TM_n^{(p)}(T)=\frac{n+p+1}{n+1}M_{n+1}^{(p)}(T)-\frac{p}{n+1}I\,,
\end{equation}
\begin{equation}\label{3}
 \frac{n+p+1}{n+1}
M_{n+1}^{(p)}(T)-M_n^{(p)}(T)=\frac{p}{n+1} M_{n+1}^{(p-1)}(T)\,.
\end{equation}
It is easy to see by the relations \eqref{1} and \eqref{2} before that the regularity of $\{M_n^{(p)}(T)\}$ just means the ergodicity of $\{M_n^{(p+1)}(T)\}$. In addition, if $\{M_n^{(p)}(T)\}$ is ergodic it is also regular (by using \eqref{3}), hence each bounded mean $\{M_n^{(p)}(T)\}$ satisfies \eqref{reg} with $m=0$ and $n_0=1$.

It is interesting to remark that these means satisfy the regularity assumption in the operator norm when they are bounded on $\overline{\R(T-I)^m}$.

Another example of a sequence $\{T_n\}$  in $\kappa(T)$  is given by the {\it binomial means}
\begin{equation}\label{dykema}
T_n= 2^{-n}\sum_{k=0}^n\binom{n}{k}T^k=S^n.
\end{equation}
where $S=\frac{1}{2}(T+I)$. Clearly, $\R(T-I)^m=\R(S-I)^m$ for $m\ge 0$, and $\{T_n\}$ satisfies \eqref{reg} for some integer $n_0\ge1$ if and only if it is ergodic. For example, as it is proved in [DyS], every contraction on a Hilbert space will satisfy \eqref{reg}.

Clearly, one may consider also continuous summability methods, for example, the {\it Abel means} defined as
\begin{equation}\label{abelsum}
A(T,r)=(1-r)\sum_{j\ge 0}r^jT^j=(1-r)(I-rT)^{-1}\,,\quad 0<r<1\,,\end{equation}
where $T\in \B$ has spectral radius at most one. Using the equality
$$(I-rT)(I-\rho T)^{-1}= I+(\rho-r)(I-\rho T)^{-1}\,,$$
it is easy to verify that the above sums are bounded if and only if the sequence $\{T_n\}$ with $T_n=A(T,1-\frac1{n})$ is bounded. Moreover, since $\frac{A(T,r)}{1-r}=(1-rT)^{-1}$ we have
$$T_n(T-I)=\frac{1}{n-1}(T_n-I)$$
and it is again easy to see that $$
TT_n-T_{n+1}=\frac{1}{(n+1)^2}\left(\frac{n(n+3)}{n-1}T_n+I\right).
$$
Hence the sequence $\{T_n\}$ is ergodic if and only if it is regular, and this equivalently means that $\frac{1}{n}T_n\to 0$ strongly on $\x$. In this case $\{T_n\}$ also satisfies \eqref{reg} for any positive integer $n_0$ (with $m=0$) even in operator norm if it is bounded.

Finally, one can deal with more complicated continuous summability methods as well. If $F$ is analytic in the unit disc $\D$ and has positive Taylor coefficients, then for any operator $T\in \B$ with spectral radius at most $1$, and any strictly increasing sequence $\{r_n\}$ tending to $1$, the identity
 \begin{equation}\label{powerseries}T_n=\frac1{F(r_n)}F(r_nT)\end{equation}
 defines a sequence in $\kappa(T)$. Of course, the regularity of such sequences is more involved and depends on the Taylor coefficients of $F$, as well as on the choice of $\{r_n\}$.

\subsection{The backward iterate}
Given a sequence $\{T_n\}$ in $\kappa(T)$ with $T_n=\sum_{j\ge 0}t_{nj}T^j$, where
\begin{equation}\label{backit1}t_{n0}\ne 1\,,\quad\sum_{j\ge 1}jt_{nj}<\infty\,\end{equation}
for all $n\in \N$, its {\it backward iterate} $\{T_n^{(-1)}\}$ is defined by
\begin{equation}\label{backit2}T_n^{(-1)}=\sum_{k\ge 0}s_{nk}T^k\,,\quad \text{where}\quad s_{nk}=\frac{\sum_{j\ge k+1}t_{nj}}{\sum_{j\ge 1}jt_{nj}}\,.\end{equation}
Then $T_n^{(-1)}$ belongs to $\kappa(T)$ as well, and satisfies the identity
\begin{equation}\label{backitid}T_n^{(-1)}(T-I)=\left(\sum_{j\ge 1}jt_{nj}\right)^{-1}(T_n-I)\,.\end{equation}

It is a simple exercise to verify that if $T_n=M_n^{(p)}(T)$ then by \eqref{1}, for $n>1$, we have $T_n^{(-1)}=M_{n-1}^{(p+1)}(T)$.

Another direct computation reveals that for the discretized Abel means introduced before $T_n^{(-1)}$ coincides with $T_n$.

Also, if $T_n$ is a sequence of the form \eqref{powerseries} then $T_n^{(-1)}=G_n(T)$, where
\begin{equation}\label{powerseriesback} G_n(z)=\frac{F(r_nz)-F(r_n)}{r_nF'(r_n)(z-1)}\,.\end{equation}

Similarly, for the sequence $\{T_n\}$ given by \eqref{zweier} we have for $n\ge 2$,
\begin{eqnarray}\label{ec36}
T_n^{(-1)}&=&\frac{2}{2n-1}\left(\sum_{k=0}^{n-2}T^k +\frac1{2}T^{n-1}\right)\\
&=&\frac{2}{2n-1}(nM_{n-1}(T)-\frac{1}{2}T^{n-1}).\notag
\end{eqnarray}

Also, it is easy to verify that $T_n^{(-1)}$ is ergodic if and only if $\{M_n(T)\}$ converges strongly to zero on $\R(T^2-I)$, and it is regular if and only if $\frac{1}{n}T_n^{(-1)}\to 0$ strongly on $\x$. In this case, one can see that if $\{T_n^{(-1)}\}$ is ergodic then it is also regular, but not conversely. In particular, the ergodicity, or regularity of $\{T_n^{(-1)}\}$ are assured by the ergodicity of $\{M_n(T)\}$, respectively of $\{M_n^{(2)}(T)\}$. In fact, $\{M_n(T)\}$ is ergodic if and only if $\{T_n^{(-1)}\}$ is ergodic and $\frac{1}{n}T^n(T-I)x\to 0$ for $x\in \x$. On the other hand, $\{M_n^{(2)}(T)\}$ is ergodic if and only if $\{T_n^{(-1)}\}$ is regular and $\frac{1}{n^2}T^nx\to 0$ for $x\in \x$. Remark that in these last remarks, the conditions of ergodicity or regularity of $\{T_n^{(-1)}\}$ are superfluous, if the means $M_n(T)$ or $M_n^{(2)}(T)$ are bounded, respectively.

Finally, we note that in certain cases the regularity condition  can imply ergodicity. Such an example is the sequence $\{T_n\}$ with $T_n=M_n(T)^2$. We omit  the details.

\section{Operators from linear dynamics}

In this section we refer to two classes of operators related to linear dynamics, which have been considered in [AB], [BM], [BC], [GLM], [He] and [K\'e].

 Recall that an operator $T\in \B$ is called {\it supercyclic} if there exists $x\in \x$ such that the set $\{\lambda T^nx: \lambda\in \C, \hspace*{1mm} n\in \N\}$ is dense in $\x$. In this case, any
such vector $x$ is called a {\it supercyclic vector} for $T$.

According to the terminology of [BC] we say that $T\in \B$ is {\it (weakly) hypercyclic with support $N$} ($N\in \N\setminus\{0\}$) if
there exists $x_0\in \x$ such that the set
\begin{equation}\label{hypercyclic}
\{ T^{k_1}x_0+T^{k_2}x_0+...+T^{k_N}x_0: k_1,...,k_N\ge 0\}
\end{equation}
is (weakly) dense in $\x$. Clearly, these concepts only make sense when $\x$ is separable.
The motivation for considering such conditions is   the Ansari-Burdon theorem ([AB], Theorem 2.2) which asserts that any  supercyclic power bounded operator $T$ must satisfy the condition
$$\lim_{n\to\infty}\|T^nx\|=0\,,\quad x\in \x\,.$$
Let us start with the weakly hypercyclic case and record the following simple observation.

\begin{lemma}\label{le31}
 Let $\y$  be a separable Banach space and let  $V\in \mathcal{B} (\y)$ be a contraction. If  $\lambda V$ is
weakly hypercyclic with support $N$, for some $\lambda\in \C$ with $|\lambda|\le1$,  then $\y=\{0\}$.
\end{lemma}

\begin{proof}  If $x_0\in \y$, the set in \eqref{hypercyclic} is contained in the ball centered at the origin and of radius $N\|x_0\|$ (as $\|V\|\le 1$), hence it cannot be weakly dense in $\y$, unless $\y=\{0\}$.
\end{proof}

Throughout in what follows, for a bounded sequence $\{T_n\}$ in $\kappa(T)$ we let $\gamma$ the seminorm from \eqref{gamma} with $m=0$, and consider the operator $V_\gamma$ obtained by an application of Proposition \ref{secondobs} to this seminorm. As a consequence of Lemma \ref{le31} we have

\begin{corollary}\label{hyper}
Assume  that  $\lambda T$ is weakly  hypercyclic with support $N$, for some $\lambda\in \C$ with $|\lambda|\le 1$. If the sequence $\{T_n\}$ in $\kappa(T)$ is bounded and the operator  $V_\gamma$  is a contraction then $\{T_n\}$ converges strongly to zero on $\x$. In particular, this holds when $\{T_n\}$ satisfies  \eqref{reg} with $m=0$ and $n_0\ge 0$.
\end{corollary}

\begin{proof} The assumption on $\lambda T$ assures for the operators $V_{\gamma}$ and $Q_{\gamma}$ given by (the proof of) Proposition \ref{secondobs} that if $x_0$ is the vector in \eqref{hypercyclic}, the set
\begin{align*}&\{Q_\gamma(\lambda T)^{k_1}x_0+Q_\gamma(\lambda T)^{k_2}x_0+...+Q_\gamma(\lambda T)^{k_N}x_0: k_1,...,k_N\ge
0\}\\&=\{ \lambda V_\gamma)^{k_1}Q_\gamma x_0+(\lambda V_\gamma)^{k_2}Q_\gamma x_0+...+(\lambda V_\gamma)^{k_N}
Q_\gamma x_0: k_1,...,k_N\ge
0\}\end{align*}
is weakly dense in $\x_\gamma$, because $Q_\gamma^*$ maps $\x_\gamma^*$ into $\x^*$. Thus    $\lambda V_\gamma$ is   hypercyclic with support $N$, and by the previous lemma we must have $\x_\gamma=\{0\}$ which means $T_nx \to 0$ for $x\in \x$.
\end{proof}

Let us now turn to supercyclic operators $T$, where  the situation is more subtle. This is due to the fact that for a supercyclic operator $T$, its adjoint $T^*$ may have eigenvalues.  A simple consequence of the definition (see also [BM]) shows that such an eigenvalue is unique, and the corresponding eigenspace has dimension one. Due to this fact  the Ansari-Bourdon theorem does not extend directly to more general means, in particular Corollary \ref{hyper} does not necessarily hold for supercyclic operators. In fact, in Subsection 5.1 we shall construct a Ces\`aro (even uniformly Kreiss) bounded and supercyclic operator $T$ such that both $T$ and $T^*$ have the eigenvalue $1$, hence $\{M_n(T)\}$ does not converge strongly to zero. It could also happens that $T^*$ has an eigenvalue, but $T$ does not.

The following theorem describes the possibilities that can occur in the general context, where the condition \eqref{reg} is essential.

\begin{theorem}\label{supercyclic}
Let $T\in \B$ be supercyclic and assume that the sequence $\{T_n\}\subset \kappa(T)$ is bounded and satisfies condition \eqref{reg} for some $n_0\ge 0$.
If $T^*$ has no eigenvalues on the unit circle then $\{T_n\}$ converges strongly to zero on $\x$, that is $\x_0=\x$.

If $T^*$ has the (unique)  eigenvalue $\mu \in \T$  then :\\
(i) $\{T_ny\}$ converges to zero for some $y\in \x\setminus\overline{\R(T-\mu I)}$ if and only if $\x_0=\x$. In this case $\mu\ne 1$. \\
(ii) $\{T_ny\}$ converges to a nonzero limit for some $y\in \x\setminus\overline{\R(T-\mu I)}$ if and only if $\x_0\ne \x_c=\x$. This is further equivalent to the fact that both $T$ and $T^*$ have the eigenvalue $1$.\\
(iii) $\{T_ny\}$ diverges  for some (and hence for all) $y\in \x\setminus\overline{\R(T-\mu I)}$ if and only if $\x_0=\x_c\ne \x$.
 \end{theorem}

\begin{proof}  Let $\gamma$ be the seminorm given by $(2.2)$ with $m=0$. From the fact that $T$ is supercyclic it follows immediately that the isometry $V_{\gamma}$ on $\x_{\gamma}$ given by condition \eqref{reg} and Theorem \ref{regobs} is supercyclic, as well. Thus, by the Ansari-Bourdon theorem mentioned before, we have $\dim \x_{\gamma} \le 1$. Clearly, $\x_{\gamma}=\{0\}$ means $\x_0=\x$ and this certainly occurs when $T^*$ has not eigenvalues on $\T$, as we see below.

Suppose that $\dim \x_{\gamma}=1$ . Then there exists $\mu\in \T$ such that
$$V_\gamma [x]=\mu[x]\,,\quad [x]=x+\n(\gamma), \quad x\in \x\,,$$
and since $\n(\gamma)=\x_0$, we obtain that $\{T_n(T-\mu I)\}$ converges strongly to zero, i.e. $\{T_n\}$ converges strongly to zero on $\overline{\R(T-\mu I)}$. If $T^*$ has no eigenvalues on $\T$ we have  $\overline{\R(T-\mu I)}=\x$ and this is a contradiction which shows that $\x_0=\x$.\\
Suppose that $T^*$ has the eigenvalue $\mu\in \T$. As pointed out above,
the corresponding eigenspace has dimension  one. Thus for every  $y\in \x\setminus\overline{\R(T-\mu I)}$
 there exists $f\in X^*,~T^*f=\mu f\neq 0$, such that $$x-f(x)y\in \overline{\R(T-\mu I)}\,,\quad x\in \x\,.$$
With respect to the direct sum decomposition $\x=\overline{\R(T-\mu I)}\oplus \C y$ the operator  $T$  has the matrix representation
$$T=
\begin{pmatrix}
T|\overline{\R(T-\mu I)} & (T-\mu I)|\C y \\
0 & \mu I_2
\end{pmatrix},
$$
where $I_2$ denotes the identity on $\C y$. A simple computation shows that
$$T^n=
\begin{pmatrix}
T^n|\overline{\R(T-\mu I)} & (T^n-\mu^n I)|\C y \\
0 & \mu^n I_2
\end{pmatrix},
$$
so that, if $$T_n=\sum_{j\ge 0} t_{nj}T^j\,,$$ and we denote $$\tau_n(\mu)= \sum_{j\ge 0} t_{nj}\mu^j\,,$$
we obtain
$$T_n=
\begin{pmatrix}
T_n|\overline{\R(T-\mu I)} & (T_n-\tau_n(\mu) I)|\C y \\
0 & \tau_n(\mu) I_2
\end{pmatrix}.
$$

(i) If we can choose $y$ as above such that $T_ny\to 0$, then from
\begin{equation*}\label{taun}f(T_n y)=\tau_n(\mu)f(y),\end{equation*}
we see that $\tau_n(\mu)\to 0$ and we obtain that  $\{T_n\}$ converges strongly to zero on $\x$. The converse is obvious. Note also that  since  $\tau_n(\mu)\to 0$ and $\tau_n(1)=1$, we must have $\mu\ne 1$.

(ii) We show  the equivalence of the first and the third assertion.   If we can choose $y$ such that $T_ny\to y_0\ne 0$, then by \eqref{reg} we have
$$Ty_0=\lim_{n\to\infty}T_nTy=\lim_{n\to\infty}T_{n+n_0}y=y_0\,.$$ from the fact that $T_n(T-\mu I)y\to 0$ we deduce that
$Ty_0=\mu y_0$, hence $\mu=1$  and from above, $1$ is an eigenvalue for $T$. Conversely, if $\mu=1$ and $(T-I)y_0=0, y_0\ne 0$, then
$$T_ny_0=\tau_n(1)y_0=y_0\,,$$
i.e. $y_0\in  \x\setminus\overline{\R(T-\mu I)}$ and $\{T_ny_0\}$ converges to $y_0$. Moreover, this assumption implies
$$T_n=
\begin{pmatrix}
T_n|\overline{\R(T- I)} & 0 \\
0 &  I_2
\end{pmatrix} \to
\begin{pmatrix}
0 & 0 \\
0 &  I_2
\end{pmatrix}\,,
$$
in the strong operator topology, that is $\x_0\ne \x_c=\x$.  If the last relation holds it is obvious  that we can choose $y$ as in the statement.\\
(iii) follows immediately from the fact that   $\{T_n\}$ converges strongly to zero on $\overline{\R(T-\mu I)}$ and that this space has codimension one.
\end{proof}

If one of the alternatives (i) or (ii) occur then $\{T_n(T-I)\}$ converges strongly to zero, that is $\{T_n\}$ satisfies the regularity assumption \eqref{reg} with $m=n_0=0$. Some interesting examples appear in the study of rotational invariant conditions, like Kreiss, or uniform Kreiss boundedness, and will be discussed in the last section.

Within this class of operators the situation becomes more transparent, as the following result shows.

\begin{corollary}\label{supercyclicse}
Let $T\in \B$ be supercyclic and we assume that the sequence $\{T_n\}\subset \kappa(T)$ is bounded and ergodic.
If $T^*$ has no eigenvalues on $\T$ then $\{T_n\}$ converges strongly to zero on $\x$ that is $\x_0=\x$.

If $T^*$ has the (unique)  eigenvalue $\mu \in \T$  then :\\
(i) $\x_0=\x$ if and only if $\mu\ne 1$. \\
(ii)  $\x_0\ne \x_c=\x$ if and only if both $T$ and $T^*$ have the eigenvalue $1$.\\
(iii)  $\x_0=\x_c\ne \x$ if and only if $\sigma_p(T^*)\setminus\sigma_p(T)=\{1\}$.
 \end{corollary}
\begin{proof}
In view of Theorem \ref{supercyclic} we only need to prove that if $T^*$ has the unique  eigenvalue $\mu \in \T\setminus\{1\}$  then $\x_0=\x$. This is however obvious, since in this case $\overline{\R(T-I)}=\x=\x_0$.
\end{proof}

One reason why the Ansari-Bourdon theorem holds is a result in [BM] which asserts that if $T^*$ has the eigenvalue $\mu$ then
 $T|\overline{\R(T-\mu I)}$ is hypercyclic, hence $T$ cannot be power bounded.  However,  even in the context of power boundedness our method based on the intertwining of $T$ with $V_{\gamma}$ (as in the previous proof) leads to an improvement of the Ansari-Bourdon result.  For a supercyclic power bounded operator $T$ it shows that every sequence in $\kappa(T)$ which satisfies \eqref{reg}  must converge strongly  to zero.

Despite the problems that occur in the general context considered here, one can prove the following interesting extension of the Ansari-Bourdon result.

\begin{theorem}\label{genAB}
Let $T\in \B$ be supercyclic and assume that the sequence $\{T_n\}\subset \kappa(T)$ is bounded and satisfies the regularity assumption \eqref{reg} for some integer $n_0\ge 0$. If there exists a supercyclic vector $x_0\in \x$ satisfying one of the following conditions
$$\inf_{n\in \mathbb{N}}\|T^nx_0\|=0\,,\leqno(i)$$
or $$\sup_{n\in \mathbb{N}}\|T^nx_0\|<\infty\,,\leqno(ii)$$
then $\{T_n\}$ converges strongly to zero on $\x$.
\end{theorem}
\begin{proof}  If (i) holds, there is  a subsequence $\{m_j\}\subset \N$ with $T^{m_j}x_0\to 0$ as $j\to \infty$. If $Q_\gamma$ and $V_\gamma$ are the operators associated to  the seminorm  $\gamma$  given by $(2.2)$ with $m=0$, from the relation $V_{\gamma} ^{m_j}Q_{\gamma} x_0=Q_{\gamma} T^{m_j}x_0\to 0$ it follows that $Q_{\gamma} x_0=0$ ($V_{\gamma}$ being an isometry). This gives $Q_{\gamma}=0$, since $x_0$  is a supercyclic vector for $T$. Hence $\x_0=\n(Q_{\gamma})=\x$ which means that  $T_n\to 0$ strongly on $\x$.

If (ii) holds, we can use (i) to assume, without loss of generality, that $\alpha :=\inf_{n\in \N} \|T^nx_0\|>0$. Let $0\neq x\in \x$ and $\{\lambda _k\}\subset \C \setminus \{0\}$, $\{n_k\}\subset \N$ be such that $x=\lim_{k\to \infty} \lambda _k T^{n_k}x_0$. Then
$$
\|x\|=\lim_{k\to \infty} |\lambda _k|\|T^{n_k}x_0\|\ge \alpha \limsup_{k\to \infty} |\lambda _k|,
$$
therefore the sequence $\{\lambda _k\}$ is bounded, and so it contains a convergent subsequence, let's say $\lambda _{k_j} \to \lambda $, $j\to \infty$. Obviously, one has $\lambda \neq 0$ because $\{T^{n_{k_j}}x_0\}$ is bounded and $\lim_{j\to \infty} \lambda _{k_j} T^{n_{k_j}}x_0=x \neq 0$. Then for $m\in \N$ we have that
$$
\|T^mx\|=\lim_{j\to \infty} |\lambda _{k_j}|\|T^mT^{n_{k_j}}x_0\|\le c|\lambda |.
$$
Hence $T$ is power bounded and so $T^m\to 0$ strongly. In turn, this gives
$$
\|T_nx\|=\lim_{j\to \infty} \|\lambda _{k_j}T_nT^{n_{k_j}}x_0\|\le |\lambda |\sup_{m\in \N} \|T_m\|\lim_{j\to \infty} \|T^{n_{k_j}}x_0\|=0,
$$
and consequently $T_nx\to 0$ for any $x\in \x$.
\end{proof}

 \section{Conditions invariant to rotations}
Given  a sequence $\{T_n\}$ in $\kappa(T)$ and $\lambda\in \C$, we denote by $T_{n\lambda}$ the sequence obtained by applying the same convex combinations to the operator $\lambda T$, that is
 \begin{equation}\label{rotinv} T_n=\sum_{j\ge 0}t_{nj}T^j\, \Longrightarrow  T_{n\lambda}=\sum_{j\ge 0}t_{nj}\lambda^j T^j
 \,,\end{equation}
 where $t_{nj}\ge 0,~\sum_{j\ge 0}t_{nj}=1$. We are interested in the effect of growth restrictions for $T_{n\lambda}$, that are uniform in $\lambda \in \T$.

Note that if $T_n=T^n$,  then $T_{n\lambda}=\lambda^nT_n,~\lambda\in \T$. If $T_n=\frac1{2}(T^{n-1}+T^n)$, it easy to verify that
$\{T_{n\lambda}\}$ are uniformly bounded on the unit circle, if and only if $T$ is power bounded. The  situation become much more interesting for
other  sequences.
For example,  if $\{T_n\}$ is the sequence of discretized Abel sums
\begin{equation}\label{abelsumd}T_n=\frac1{n}\sum_{j=0}^\infty (1-\frac1{n})^jT^j\,,\end{equation} then the uniform boundedness of the sequences $\{T_{n\lambda}\},~\lambda\in \T$, is equivalent to the {\it Kreiss boundedness condition}
$$\|(T-\lambda I)^{-1}\|\le \frac{C}{|\lambda|-1}\,,\quad |\lambda|>1\,,$$
for some constant $C>0$.

\subsection{A generalized Nevanlinna theorem}
O. Nevanlinna [N] proved the interesting result that for operators $T\in  \B$ satisfying the Kreiss boundedness condition,
and with $\sigma(T)\cap \T$ of arclength measure zero, one has $$\|T^n\|=o(n), \quad n\to \infty\,.$$

The aim of the present subsection is to show that this type of estimate holds in a very general context and can be derived from rotational invariant conditions.  This requires a different approach which makes use of the  backward iterate $T_n^{(-1)}$ defined in \eqref{backit2}.

\begin{theorem}\label{nevanlinna} Let $\{T_n\}$ be a sequence in $\kappa(T)$  with $T_n=\sum_{j\ge 0}t_{nj}T^j$, that satisfies \eqref{backit1} and
$$\lim_{n\to\infty}\sum_{j\ge 1}jt_{nj}=\infty\,.$$
 If $\{w_n\}$ is an increasing sequence of positive numbers such that
$$\|T_{n\lambda}\|+\|T_{n\lambda}^{(-1)}\|\le w_n\,,$$
for all $\lambda \in \T$, and $\sigma(T)\cap\T$ has arclength measure zero, then for every strictly increasing sequence of positive integers $\{m_n\}$, we have $$\|T^n\|\frac{\sum_{j\ge n+1}t_{m_nj}}{\sum_{j\ge 1}jt_{m_nj}}=o(w_{m_n})\,,\quad n\to\infty\,.$$
\end{theorem}
\begin{proof} We obviously have for $\lambda\in \T$,  $$T_{n\lambda}^{(-1)}(\lambda T-I)=\left(\sum_{j\ge 1}jt_{nj}\right)^{-1}( T_{n\lambda}-I)\,,$$
hence, by assumption we obtain
$$\|T_{n\lambda}^{(-1)}| \R(\lambda T-I)\|=o(w_n), \quad n\to \infty\,.$$
Since $ \R(\lambda T-I)=X$  a.e. on $\T$ with respect to arclength measure, it follows that
\begin{equation}\label{domconv}\| T_{n\lambda}^{(-1)}\|=o(w_n)\,,\end{equation}
a.e. on $\T$. With the notation in \eqref{backit2} we  also have
$$w_{m_n}^{-1}s_{m_nn}T^n=w_{m_n}^{-1}\int_0^{2\pi}T^{(-1)}_{m_n e^{it}}e^{-int}\frac{{\rm d}t}{2\pi}\,,$$
in the sense of Bochner integral, so that
\begin{equation}\label{backitlast}w_{m_n}^{-1}s_{m_nn}\|T^n\|\le w_{m_n}^{-1}\int_0^{2\pi}\|T^{(-1)}_{m_n e^{it}}\|\frac{{\rm d}t}{2\pi}\,.\end{equation}
Again by assumption we have that
$$\sup_{t\in [0,2\pi]} w_{m_n}^{-1}\|T^{(-1)}_{m_n e^{it}}\|=O(1), \quad n\to \infty \,,$$
and by \eqref{domconv} we can apply the dominated convergence theorem in \eqref{backitlast} to obtain the conclusion.
\end{proof}

\begin{remark} {\rm The reason for considering arbitrary increasing sequences of integers $\{m_n\}$ in Theorem 4.1
is that for  a general sequence $\{T_n\}$ in $\kappa(T)$, the $n$-th power $T^n$ may not appear in the $n$-term $T_n$. Moreover, different choices of $\{m_n\}$ may lead to different estimates of $\|T^n\|$. Both aspects occur in the case of Ces\`aro means.}\end{remark}

 A very natural class of sequences $\{T_n\}$ for which Theorem \ref{nevanlinna} applies directly, are those generated by power series with positive coefficients, as defined in \eqref{powerseries}, since their backward iterates are easily determined by \eqref{powerseriesback}. However, in order to avoid technicalities we shall only consider the particular case of the discretized Abel sums for which we have $T_n^{(-1)}=T_n$. As a direct consequence we find the following {\it extended version of Nevanlinna's theorem}.

\begin{corollary}\label{Abel}
Let $u:(1,\infty)\to (0,\infty)$ be a decreasing function.  If $T\in \B$ satisfies the condition
$$\|(T-\lambda I)^{-1}\|\le \frac{u(|\lambda|)}{|\lambda|-1}\,,\quad |\lambda|>1\,,$$
and $\sigma(T)\cap \T$ has arclength measure zero, then
$$
\|T^n\|=o\left(nu\left(\frac{n}{n-1}\right)\right), \quad n\to \infty.
$$
\end{corollary}
\begin{proof} We have that
the Abel means \eqref{abelsumd} are in $\kappa(T)$ and satisfy $$\|T_n\|\le u\left(\frac{n}{n-1}\right)\,.$$
By the observation preceding the corollary, $T_n^{(-1)}$ satisfy the same inequality and the result follows by an application of
 Theorem \ref{nevanlinna}.
\end{proof}
With a specific choice of the function $u$ above we obtain :

\begin{corollary}\label{nevanlinna3} If $T\in \B$ satisfies for some $r\ge 0$ the condition
$$\sup_{|\lambda|>1}\frac{(|\lambda|-1)^{r+1}}{|\lambda |^r}\|(T-\lambda I)^{-1}\|<\infty\,,$$
and  $\sigma(T)\cap\T$ has arclength measure zero, then
$$\|T^n\|=o(n^{r+1})\,,\quad n\to\infty\,.$$

\end{corollary}

The next corollary provides only spectral conditions (on the resolvent function of $T$) in order for $\{M_n^{(r)}(T)\}$ to be uniformly convergent to the ergodic projection, for $r\ge 1$.

\begin{corollary}\label{uniformCesaro}
Let $T\in \B$. If $1$ is a simple pole of the resolvent of $T$, $\sigma(T)\cap \T$ has arclength measure zero and for an integer $r\ge 1$ we have
$$
\sup_{|\lambda|>1}\frac{(|\lambda|-1)^{r}}{|\lambda |^{r-1}}\|(T-\lambda I)^{-1}\|<\infty, \quad \sup_{\lambda >1} (\lambda -1)\|(T-\lambda I)^{-1}\|<\infty \,,
$$
then $$\|M_n^{(r)}(T)-P_T\|\to 0, \hspace*{1mm} {\rm as}\hspace*{1mm} n\to \infty.$$
\end{corollary}

\begin{proof}
This is obtained by using Corollary \ref{nevanlinna3}, Theorem 6 in [Hi], Theorem 1 [Ed] and the uniform Abel ergodic theorem in [LSS].
\end{proof}

We should point out here that the hypothesis that $\sigma(T)\cap\T$ has
arclength measure zero cannot be relaxed, in general (see [N] and Subsection 5.3 below).

\subsection{Ces\`aro means and resolvent estimates}

As mentioned in Section 2, if $T_n=M_n^{(p)}(T)$, then $T_n^{(-1)}=M_{n-1}^{p+1}(T)$. Then
 Theorem \ref{nevanlinna} applies to Ces\`aro means in the following way.
\begin{corollary}\label{nevanlinna1} Let $\{w_n\}$ be an increasing sequence of positive numbers. If $T\in \B$ satisfies the condition
$$\sup_{n\in \N\atop \lambda\in \T}w_n^{-1}\|M_n^{(p)}(\lambda T)\|<\infty\,,$$
for some integer $p\ge 1$,
and $\sigma(T)\cap\T$ has arclength measure zero, then
$$\|T^n\|=o(nw_{pn})\,,\quad n\to\infty\,.$$
\end{corollary}
\begin{proof}  By \eqref{0} we have that
$$\sup_{n\in \N\atop \lambda\in \T}w_n^{-1}\|M_n^{(p+1)}(\lambda T)\|<\infty\,,$$
hence, by the remarks above, we can apply Theorem \ref{nevanlinna} to $T_n=M_n^{(p)}(T)$ with $m_n=pn$. From the second equality in \eqref{0} we have $$s_{m_nn}=\frac{\sum_{j\ge n+1}t_{m_nj}}{\sum_{j\ge 1}jt_{m_nj}}=\frac{p+1}{pn+1}\left(\frac{p-1}{p}\right)^p\alpha_n \,,$$
where $\alpha _n \to 1^+$. The conclusion follows from  Theorem \ref{nevanlinna}.
\end{proof}

For the special choice $w_n=n^r,~n\ge 1$, with $r\ge 0$ fixed,  Corollary \ref{nevanlinna1} yields the implication
\begin{equation}\label{nevanlinna2}
\sup_{n\ge 1 \atop \lambda\in \T}n^{-r}\|M_n^{(p)}(\lambda T)\|<\infty\Longrightarrow \|T^n\|=o(n^{r+1}), \quad n\to \infty\,,\end{equation}
whenever  $p\ge 1$ and $\sigma(T)\cap\T$ has arclength measure zero. It turns out that such ''regular'' growth restrictions for the Ces\`aro means can be reformulated in terms of resolvent estimates. This
has been first proved by  Strikwerda and Wade [SW] in the case when $p=2$ and $r=0$. The Ces\`aro means of order one are related to the partial sums of the Taylor expansion at infinity  of the resolvent function of $T$. The uniform boundedness of these partial sums in the exterior of the unit disc is called  the {\it uniform Kreiss boundedness condition} (see [MSZ] and [GH] for the boundedness only on the real line) and is equivalent to the uniform boundedness of $\{M_n(\lambda T)\}$ on the unit circle. These last two conditions are more restrictive.
The result below provides an extension of the theorems we just mentioned.

\begin{theorem}\label{kreisstype}
Given  $T\in  \B$ we have:\\
(i) If $p\ge 2$ and $r\ge 0$ then
$$\sup_{n\ge 1 \atop \lambda\in \T}n^{-r}\|M_n^{(p)}(\lambda T)\|<\infty \,\,\Longleftrightarrow \,\,
\sup_{ |\lambda|>1}\frac{(|\lambda|-1)^{r+1}}{|\lambda |^r}\|(T-\lambda I)^{-1}\|<\infty\,,$$
(ii) If $p=1$ and $r\ge 0$ then
$$\sup_{n\ge 1 \atop \lambda\in \T}n^{-r}\|M_n(\lambda T)\|<\infty \,\,\Longleftrightarrow \,\,
\sup_{n\in \N\atop |\lambda|>1}\frac{(|\lambda|-1)^{r+1}}{|\lambda |^r} \left\|\sum_{k=0}^n\lambda^{-k-1}T^k \right\|<\infty\,.$$
\end{theorem}
\begin {proof} $(\Longrightarrow)$ The argument uses the ideas in [SW]. We start with the identity
 proved in [SW], formula (6.3), for $0<\rho<1$ and $\lambda\in \T$, namely
$$(I-\rho\lambda T)^{-1}=(1-\rho)^p\sum_{n=0}^\infty\binom{n+p}{p}M_n^{(p)}(\lambda T)\rho^n\,.$$
For $p\ge2$ the resolvent estimate follows directly from
$$\sum_{n=0}^\infty\binom{n+p}{p}n^r\rho^n\le C(1-\rho)^{-p-r-1}\,.$$
By a straightforward argument we see that the above equality holds for all $\rho\in \D$ and by comparing the coefficients of $\rho^k$ we obtain for $p=1$,
$$\sum_{k=0}^n\rho^k\lambda^kT^k=(1-\rho)\sum_{k=0}^{n-1}(k+1)M_k(\lambda T)\rho^k +(n+1)M_n(\lambda T)\rho^n\,.$$
In this case the desired estimate follows from the inequalities
 $$\sum_{n=0}^{n-1}(k+1)^{r+1}\rho^k\le C(1-\rho)^{-r-2}\,,\quad (1-\rho)^{r+1}\rho^n\le C'(n+1)^{-r-1}\,,\,$$
 for $ 0<\rho<1$.
 \smallskip

$(\Longleftarrow)$  We begin with (i) and assume first that $p=2$. Use \eqref{0} to write for  $0<\rho<1$, $\lambda\in \T$,
$$\frac{ (n+1)(n+2)}{2} M_n^{(2)}(\lambda T)=\int_0^{2\pi}(I-\rho\lambda e^{it}T)^{-1}\sum_{j=0}^{n}(n-j+1)e^{-ijt}\rho^{-j}\frac{{\rm d}t}{2\pi}\,,$$
so that
\begin{align*}\left\|\frac{ (n+1)(n+2)}{2} M_n^{(2)}(\lambda T)\right\|&\le \frac{C}{(1-\rho)^{r+1}}\rho^{-n} \int_0^{2\pi}\frac{1}{|1-\rho e^{-it}|^2}\frac{{\rm d}t}{2\pi}\\&= \frac{C}{(1-\rho)^{r+1}(1-\rho ^2)}\rho ^{-n}\\ &\le \frac{C}{(1-\rho)^{r+2}}\rho^{-n}
\,.\end{align*}
If we now choose $\rho=1-\frac1{n}$ we obtain the estimate for $M_n^{(2)}(T)$  in (i) for all $r\ge 0$. For $p>2$, this  follows immediately from the first equality in \eqref{0}.

Let us turn to the case when $p=1$. For $n\ge 1$, $\mu\in \T$ and $\lambda \in \C$ with $|\lambda |>1$  we have  by the Abel summation formula
 $$
 M_n(\mu T) = \frac{1}{n+1}\sum_{k=0}^n (\mu T)^k=\frac{1}{n+1}\sum_{k=0}^n \lambda ^{k+1}\lambda ^{-k-1}(\mu T)^k
 $$
 $$
 = \frac{1}{n+1} \left(\lambda ^{n+1} \sum_{k=0}^n\lambda ^{-k-1}(\mu T)^k+(1-\lambda) \sum_{k=0}^{n-1} \lambda ^{k+1}\sum_{j=0}^k\lambda ^{-j-1}(\mu T)^j \right)
 $$
 $$
 = \frac{1}{n+1} \left(\lambda ^{n+1} \overline{\mu } \sum_{k=0}^n(\lambda \overline{\mu })^{-k-1}T^k+(1-\lambda) \sum_{k=0}^{n-1} \lambda ^{k+1}\overline{\mu }\sum_{j=0}^k (\lambda \overline {\mu})^{-j-1}T^j \right).
 $$
From the assumption
$$\sup_{n\in \N\atop |\lambda|>1}\frac{(|\lambda|-1)^{r+1}}{|\lambda |^r}\left\|\sum_{k=0}^n\lambda^{-k-1}T^k\right\|<\infty\,,$$
we obtain
 $$
 \|M_n(\mu T)\|\le C\frac{|\lambda |^{r+1}}{(|\lambda |-1)^{r+1}(n+1)}\left(|\lambda |^n+\frac{|1-\lambda| (|\lambda|^n-1)}{|\lambda |-1}\right)\,,
 $$
 and if $\lambda =1+\frac{1}{n}$,
 $$
 \|M_n(\mu T)\|\le 2^r (2e-1)Cn^r
 $$
 for $n\ge 1$ and $\mu \in \T$, which concludes the proof.
\end{proof}
 In Subsection 5.3 below we construct examples which show that the conditions in part (ii) are stronger than those in part (i) of the theorem.
In fact, by a similar argument as in the previous proof one can show that if one (any) of the conditions in $(i)$ is satisfied then
\begin{equation}\label{log}
\sup_{\lambda \in \T} \|M_n(\lambda T)\|=O(n^r\log(n)), \quad n\to \infty.
\end{equation}

On the other hand, the conditions in $(ii)$ do not imply a similar growth condition for the powers $T^n$. Indeed, let us consider $T=I-V$, where $V$ is the classical Volterra operator
$$
(Vf)(t):=\int_0^tf(s){\rm d}s, \quad f\in L^p[0,1], \quad 1\le p\le \infty, \quad p\neq 2.
$$
It is known from [MSZ] that $T$ is uniformly Kreiss bounded and $\|T^n\| \sim n^{|\frac{1}{4}-\frac{1}{2p}|},$
where $\sim$ stands for ''comparable''. In particular this gives that $\{\frac{T^n}{\log(n)}\}$ is unbounded.

\subsection{The case $p=1$}
We want to investigate the growth restrictions
\begin{equation}\label{ec48}\sup_{n\ge 1 \atop \lambda\in \T}n^{-r}\|M_n(\lambda T)\|<\infty\,,\end{equation}
in more detail. The somewhat surprising result we are going to prove below asserts that under these conditions, the conclusion of Theorem \ref{nevanlinna} in the sense of strong convergence only holds under an additional assumption. In particular, our theorem provides a partial answer to Question 3 in [MSZ].

\begin{theorem}\label{uniformKreiss} Let $T\in \B$ satisfying for some integer $r\ge 0$ the condition \eqref{ec48}, and let $Q_{\gamma}$ be the operator induced by the seminorm $\gamma$ given by \eqref{gamma} for $T_n=M_n^{(r+1)}(T)$ and $m=r+1$. If $\R(Q_{\gamma})$ is closed then
$$\|T^nx\|=o(n^{r+1})\,,\quad n\to\infty\,$$
for $x\in \x$.
\end{theorem}

\begin{proof}
Note first that a direct application of  \eqref{1} shows that under our hypothesis we have that $\{n^{-r-1}T^n\}$ is bounded, so one can define the seminorm on $\x$
$$\gamma(x)=\limsup_{n\to \infty} n^{-r-1}\|T^nx\|=\limsup_{n\to \infty} \|T_n(T-I)^{r+1}x\|, \quad x\in \x\,,$$
where $T_n=M_n^{(r+1)}(T)$. By \eqref{ec48} we have $\frac{1}{n}T_nx\to 0$ for $x\in \R(T-I)^r$, while by \eqref{2} this means that $\{T_n\}$ satisfies \eqref{reg} with $n_0=1$ and $m=r$. By applying Theorem \ref{regobs} $(ii)$ and Proposition \ref{secondobs} to $T_n$ and $\gamma$, one obtains an isometry $V_{\gamma}$ on the corresponding quotient space $\x_{\gamma}$ satisfying $V_{\gamma}Q_{\gamma}=Q_{\gamma}T$, where $Q_{\gamma}$ is the corresponding quotient map of $\x$ into $\x_{\gamma}=\overline{Q_{\gamma}\x}$. The following observation reveals a remarkable property of these operators. Namely, we have

\begin{equation}\label{mainid}\alpha=\sup_{n\ge 1\atop\lambda\in \T}\left\|\sum_{k=0}^n\lambda^kV_\gamma^kQ_{\gamma}\right\|<\infty\,.\end{equation}
Indeed, if
$$s_n(\lambda T)=\sum_{k=0}^n\lambda^k T^k\,,$$
then for $j\ge1$ we have
 $$
 T^js_n(\lambda T)=\overline{\lambda}^j[(n+j+1)M_{j+n}(\lambda T)-jM_{j-1}(\lambda T)]\,.$$
 Divide both sides by $j^{r+1}$ to obtain for $x\in \x$ and $n\ge 1$ fixed
$$\|Q_{\gamma}s_n(\lambda T)x\|=\gamma(s_n(\lambda T)x)\le 2\limsup_{m\to\infty} m^{-r}\|M_m(\lambda T)\|\|x\|\,,$$
which gives \eqref{mainid}.

We claim that \eqref{mainid} implies $Q_{\gamma}=0$, that is $\x_{\gamma}=\{0\}$, under our assumption. To see this, assume that $\R(Q_{\gamma})$ is closed and $Q_{\gamma} \neq 0$. Since $V_\gamma$ is a non zero isometry it must have an approximate eigenvalue $\zeta\in \T$. Then there exist sequences $\{x_j\}$ in $\x_\gamma$, $\{l_j\}$  in $ \x_\gamma^*$ , with
$$\|x_j\|_\gamma=1\,,\quad \lim_{j\to\infty}(V_\gamma-\zeta I)x_j=0\,,$$
and $$\|l_j\|=1\,, \quad |l_j(x_j)|>\frac1{2}\,.$$

Now by \eqref{mainid} we have for $j,n\ge 1$
$$\int_0^{2\pi}\left|\sum_{k=0}^n l_j(e^{ikt}V_{\gamma}^kx_j)\right|^2\frac{{\rm d}t}{2\pi}\le \alpha ^2\,,$$
and  Parseval's formula gives
$$\sum_{k=0}^n|l_j(V_{\gamma}^kx_j)|^2\le \alpha ^2\,.$$
Since for fixed $n$ we have
$$\lim_{j\to\infty}\sum_{k=0}^n\|V_{\gamma}^kx_j-\zeta^k x_j \|^2=0\,,$$
by letting $j\to \infty$ above we obtain
$$\limsup_{j\to\infty}\sum_{k=0}^n|l_j(x_j)|^2\le \alpha ^2\,,$$
and we arrive at the inequality $n\le 4\alpha ^2$ for all positive integers $n$, which gives a contradiction.

Thus $Q_{\gamma}=0$ and consequently, $\x_\gamma=\{0\}$, that is
$$\|T^nx\|=o(n^{r+1})\,,\quad n\to\infty\,,$$
for all $x\in \x$. This ends the proof.
  \end{proof}
 The examples in the next section show that the exponent of $n$ cannot be improved.

  As an application of this theorem, let us remark that if $T$ is supercyclic and it satisfies \eqref{ec48} then $V_{\gamma}$ is also supercyclic, hence $\dim (\x_{\gamma})\le 1$ because $V_{\gamma}$ is an isometry. So $\R(Q_{\gamma})=\x_{\gamma}$, and by Theorem \ref{uniformKreiss} we have $\x_{\gamma}=\{0\}$. In particular, if $r=0$ in \eqref{ec48} then one obtains $\|T^nx\|=o(n)$ as $n\to \infty$ for $x\in \x$. Thus we infer by Corollary \ref{supercyclicse} the following extension of the Ansari-Bourdon theorem

 \begin{corollary}\label{co49}
 If $T\in \B$ is supercyclic and uniformly Kreiss bounded such that $1 \notin \sigma_p(T^*) \setminus \sigma_p (T)$, then for every $\lambda \in \T$ the Ces\`aro means $M_n(\lambda T)$ converge strongly on $\x$.
 \end{corollary}

 Notice that in reflexive spaces, the spectral condition in this corollary is superfluous (by Corollary \ref{supercyclicse}), therefore the sequence $\{M_n(\lambda T)\}$ always converges strongly in this case, the limit being non zero if and only if $1\in \sigma_p(T) \cap \sigma_p(T^*)$. This last spectral condition can be satisfied, as we shall see in Subsection 5.1. Thus a version of Lorch's theorem [Kr, Theorem 2.1.2, p. 73] on one hand, and of Ansari-Bourdon on the other hand, is obtained for supercyclic and uniformly Kreiss bounded operators, in reflexive spaces.

A question arises which is suggested by Theorem \ref{uniformKreiss}.
\smallskip

{\it Question.} For $T\in \B$ uniformly Kreiss bounded does $\|\frac{1}{n}T^n\|$ converge to zero as $n\to \infty$?
\smallskip

  \section{Examples}

\subsection{Uniformly Kreiss bounded hypercyclic and supercyclic operators }
Consider the weighted Dirichlet spaces $D_\alpha,~\alpha>-1$ consisting of analytic functions $f(z)=\sum_{n\ge 0}f_nz^n$in the unit disc $\D$ with the property that \begin{equation}\label{dirichletnorm} \|f\|^2=\sum_{n\ge 0}(n+1)^{1-\alpha}|f_n|^2<\infty\,,\end{equation}
 Clearly, the monomials $e_n(z)=(n+1)^{(1-\alpha)/2}z^n$ form an orthonormal basis in this space. Therefore, the operator $M_z$ of multiplication by the independent variable is a forward weighted shift of the form
$$M_z e_n=\left(\frac{n+2}{n+1}\right)^{(1-\alpha)/2}\,.$$
Its adjoint $M_z^*$ is a hypercyclic backward shift, and for every $\mu\in \T$ the matrix operator
$$T_\mu=
\begin{pmatrix}
\mu M_z^* & 0 \\
0 & \mu  I_\C
\end{pmatrix}\, $$
is supercyclic on $D_\alpha\oplus \C$,  with $\sigma_p(T_\mu)\cap\sigma_p(T_\mu^*)=\{\mu\}$ (see [BM, Theorem 1.40]).

The following result is well known. For example, it can be deduced from the results in [AP]. A direct consequence of it is that $M_z^*|D_\alpha,~
\alpha>0$,  and hence all corresponding operators $T_\mu$ are uniformly Kreiss bounded.
\begin{proposition}\label{pr51}
If $\alpha>0$ then the operator $M_z$ is uniformly Kreiss bounded on $D_\alpha$.
\end{proposition}

\begin{proof} A straightforward computation based on Parseval's formula yields
$$\|f\|^2\sim |f(0)|^2+\int_\D |f'(z)|^2(1-|z|^2)^\alpha {\rm d}A(z)\,,\quad f\in D_\alpha\,,$$
where $A$ denotes the area measure.
Then for $|\lambda|=1$ we have
$$M_n(\lambda M_z)f(\zeta)=\frac{1-(\lambda\zeta)^{n+1}}{(n+1)(1-\lambda\zeta)}f(\zeta)\,.$$
Since the functions $$F_n(\lambda,\zeta)=\frac{1-(\lambda\zeta)^{n+1}}{(n+1)(1-\lambda\zeta)}$$
are uniformly  bounded in $\D\times\D$ and satisfy
$$\left|\frac{\partial}{\partial\zeta} F_n(\lambda,\zeta)\right|\le \frac{C}{|1-\lambda\zeta|}$$
for some $C>0$ and all $\lambda,\zeta\in \D$ and all positive integers $n$, we obtain
\begin{align*}\|M_n(\lambda M_z)f\|^2&\sim \frac1{(n+1)^2}\left(|f(0)|^2+\int_\D |(M_n(\lambda M_z)f)'|^2(1-|z|^2)^\alpha  {\rm d}A\right)\\&
\lesssim \frac1{(n+1)^2}|f(0)|^2+\int_\D |f(z)'|^2(1-|z|^2)^\alpha {\rm d}A\\&+\int_\D \frac{|f(z)|^2}{|1-\lambda z|^2}(1-|z|^2)^\alpha {\rm d}A \,.\end{align*}
Thus, the statement will follow if the operator $$f\mapsto \int_0^z\frac{f(\zeta){\rm d}\zeta}{1-\lambda \zeta}\,,$$
is bounded on $D_\alpha$. As pointed out before, a proof of this fact can be found in [AP].
\end{proof}

This example shows that one can has in view some non-trivial ergodic properties (as in Section 3) for supercyclic operators $T$ which are not power bounded, but with some bounded sequences in $\kappa (T)$, as well the means $M_n^{(p)}(T)$, $p\ge 1$.

\subsection{A Ces\`aro bounded 3-isometry} Given a Hilbert space $\h$ and $T\in \mathcal{B}(\h)$, we say that $T$ is an {\it $N$-isometry}, $N\in \N$, (see [AS]), if it satisfies $$\sum_{k=0}^N(-1)^k\binom{N}{k}\|T^kx\|^2=0\,,\quad x\in \h\,.$$
We shall be concerned with the cases when $N=2,3$. Our aim is to construct a Ces\`aro bounded 3-isometry $T$ on the Hilbert space $\h$ for which the sequence $\{\frac{1}{n}T^nx\}$ is bounded below for all $x\in \h\setminus\{0\}$. In particular, $\{M_n(T)x\}$ diverges for each $f\in \h\setminus\{0\}$. Such an example is unknown in literature.

In addition, we shall see that for any $x$ the sequence $\{M_n(T)x\}$ converges weakly to zero.

We start with the following simple method of constructing a 3-isometry using a given 2-isometry.
Throughout in what follows we shall denote for a polynomial $p(z)=\sum_{n\ge 0}p_nz^n$
$$\|p\|_2^2=\sum_{n\ge 0}|p_n|^2= \int_\T |p(z)|^2{\rm d}m(z)  \,,$$
where $m$ is the normalized arclength measure on the unit circle and
$$Lp(z)=\sum_{n\ge 1}p_nz^{n-1}=\frac{p(z)-p(0)}{z}\,.$$

\begin{proposition}\label{23iso} Let $T$ be a 2-isometry on the Hilbert space $\h$. Fix $x_0\in \h\setminus\{0\}$ and let $\h_1$ be the completion of the space of analytic polynomials  with respect to the norm
$$\|p\|_1^2=\|p\|_2^2+\sum_{n\ge 0}\|L^np(T)x_0\|^2\,.$$
Then the operator of multiplication by the independent variable $M_zp=zp$ extends to a bounded operator on $\h_1$ which is a 3-isometry with the property that
$$\|M_zp\|_1^2-\|p\|^2_1=\|p(T)x_0\|^2\,.$$
\end{proposition}
\begin{proof}
A straightforward computation gives $$\|M_zp\|_1^2=\|p\|_2^2+\|Tp(T)x_0\|^2+\sum_{n\ge 0}\|L^np(T)x_0\|^2\,,$$
hence $$\|M_zp\|_1^2\le (1+\|T\|^2)\|p(T)x_0\|^2$$
for all polynomials $p$. Moreover the equality
$$\|M_zp\|_1^2=\|p\|_1^2+\|Tp(T)x_0\|^2$$
follows directly from above. Since $T$ is a 2-isometry, it  satisfies for every polynomial $p$,
$$\|T^2p(T)x_0\|^2-2\|Tp(T)x_0\|^2+\|p(T)x_0\|^2=0\,.$$
This  immediately implies that
$$\|M_z^3p\|_1^2-3\|M_z^2p\|_1^2+3\|M_zp\|_1^2-\|M_zp\|_1^2=0\,,$$
and the result follows.
\end{proof}

We shall work with a specific choice of the 2-isometry $T$, namely the weighted shift $M_z$ on the Dirichlet space $D_0$ as defined in the previous subsection. Now let $x_0(z)=1-z,~z\in  \D$. We  then have the following description of the space $\h_1$ considered in Proposition \ref{23iso}.

\begin{lemma}\label{isolemma}
With the above notations we have
$$\|p\|_1^2\sim\|p\|_*^2=\|p\|_2^2+ \int_\T |p'(z)(1-z)|^2{\rm d}m(z)\,,
$$
for all analytic polynomials $p$
\end{lemma}
\begin{proof} Note that for every polynomial $p(z)=\sum_kp_kz^k$ we have
$$\|(1-M_z)L^np\|_1^2 =|p_n|^2+\sum_{k\ge n+1}|p_k-p_{k-1}|^2(k-n+1)\,,$$
so that
\begin{align}\label{normestim} \|(1-M_z)p\|_1^2&=2\|p\|_2^2+\sum_{n\ge 0}\sum_{k\ge n+1}|p_k-p_{k-1}|^2(k-n+1)\\&\nonumber
\sim \|p\|_2^2+\sum_{k\ge 1}k^2|p_k-p_{k-1}|^2\,.
\end{align}
By Parseval's formula
$$\|p\|_*^2=\|p\|_2^2+ \int_\T |p'(z)(1-z)|^2{\rm d}m(z)=|p_1|^2+\sum_{k\ge 2}|kp_k-(k-1)p_{k-1}|^2\,$$
and the result follows from the elementary inequalities
\begin{eqnarray*}
k^2|p_k-p_{k-1}|^2-2k|p_k-p_{k-1}||p_{k-1}| &\le & |kp_k-(k-1)p_{k-1}|^2\\
& \le & 2k^2|p_k-p_{k-1}|^2+2|p_{k-1}|^2.
\end{eqnarray*}

\end{proof}
An immediate consequence of the lemma is that $\h_1$ is continuously contained in the Hardy space $H^2$, that is
$$\|f\|_2\le C\|f\|_1\,,$$
for some constant $C>0$ and all $f\in \h_1$.
With the lemma in hand we can prove the following result.

\begin{theorem}\label{Cesaro3iso}
The operator $M_z$ on $\h_1$ is Ces\`aro bounded and satisfies for every $f\in \h_1$
$$\|M_z^nf\|_1\ge\|(1-M_z)f\|_2 \sqrt{\frac{n(n-1)}{2}}\,.$$

Moreover, the sequences $\{M_n(M_z)f\}$ converges weakly to zero for every $f\in \h_1$.
\end{theorem}

\begin{proof} For every polynomial $p$ we have that $$M_n(M_z)p(z)=\frac{1-z^{n+1}}{(n+1)(1-z)}p(z)\,,$$
and the inequality $|1-z^{n+1}|\le (n+1)|1-z|,~z\in \T$ easily implies that
$$\|M_n(M_z)p\|_*^2\le \|p\|_2^2+2\int_\T|p'(z)(1-z)|^2{\rm d}m(z)+2\int_\T|p(z)|^2{\rm d}m(z)\,.$$
By Lemma \ref{isolemma} and the remark following it, we obtain that $M_z|\h_1$ is Ces\`aro bounded. Moreover, it is well known and easy to prove that the norm on $D_0$ satisfies
$$\|M_z^np\|^2-\|p\|^2=n\|p\|_2^2$$
hence by Proposition \ref{23iso} we have for $f\in \h_1$
\begin{align*}\|M_z^nf\|_1^2-\|f\|_1^2 &=\sum_{k=1}^n\|M_z^kf\|_1^2-\|M_z^{k-1}f\|_1^2\\&
=\sum_{k=1}^n\|M_z^{k-1}(1-M_z)f\|_1^2\\&\ge \sum_{k=1}^n(k-1)\|(1-M_z)f\|_2^2\\&
=\frac{n(n-1)}{2}\|(1-M_z)f\|_2^2\end{align*}
and the required inequality is proved.

Now, by using the boundedness of $\{M_n(M_z)\}$, the weak convergence to zero of $\{M_n(M_z)\}$ in $\mathcal{B}(\h_1)$ one reduces to the convergence to zero of the scalar products $\langle M_n(M_z)p,q \rangle $ for all analytical polynomials $p,q$. But this easily follows from the above expression of $M_n(M_z)p$ and Lemma \ref{isolemma}. The proof is complete.
\end{proof}

\subsection{ Modified Shields examples}
This is a simple modification of the example in [Sh] showing that the hypothesis that $\sigma(T)\cap\T$ has
arclength measure zero cannot be removed form Corollary \ref{nevanlinna3}.\\

 Let $r$ be a nonnegative integer and let $\x_r$ be the Banach space of analytic functions $f$ in $\D$ with the property that $f^{(r+1)} $ belongs to the Hardy space $H^1$. The norm on $\x_r$ is given by $$\|f\|_r=\sum_{j=0}^r|f^{(j)}(0)|+\int_\T|f^{(r+1)}|{\rm d}m\,.$$
Then it is well known (see for example [Du]) that there exists $C>0$ such that
$$\int_\T|f^{(j)}|{\rm d}m\le C\|f\|_r\, \quad {\rm and}\quad \|f\|_\infty\le C\|f\|_r\,,$$
for all $j\le r$ and all $f\in \x_r$.

\begin{proposition}\label{pr55}
The operator  $T=M_z|\x_r$ has the following properties:\\
(i) For $p\ge 2$ we have that the sequences $\{n^{-r}M_n^p(\lambda T)\}$ are uniformly bounded in $\lambda \in \T$,\\
(ii) $\{n^{-r}M_n(T)\}$ is unbounded,\\
(iii) $\{n^{-r-1}T^n\}$ does not converge strongly to zero.
\end{proposition}

\begin{proof}
To see (i) we use another standard estimate (see again [Du]), to obtain for $|\lambda |<1$,
\begin{align*} \|(I-\lambda T)^{-1}f\|_r&\lesssim \sum_{j=0}^r|f^{(j)}(0)| +\sum_{j=0}^{r+1}\int_\T\frac{|f^{(j)}(z)|}{|1-\lambda z|^{j+1}}{\rm d}m(z)\\&
\lesssim\frac{\|f\|_r}{(1-|\lambda|)^{r+1}}+\int_\T\frac{|f(z)|}{|1-\lambda z|^{r+2}}{\rm d}m(z)\\&\lesssim
\frac{\|f\|_r}{(1-|\lambda|)^{r+1}}+\|f\|_\infty\int_\T\frac{1}{|1-\lambda z|^{r+2}}{\rm d}m(z)\\&\lesssim
\frac{\|f\|_r}{(1-|\lambda|)^{r+1}}\,.\end{align*}
Then the assertion follows by  an application of Theorem \ref{kreisstype}.\\
(ii) We have $M_n(T)f=F_nf$, where  $$F_n(z)=\frac{1-z^{n+1}}{(n+1)(1-z)}\,,\quad z\in \D\,.$$
We claim that for $f(z)=1$, $n^{-r}M_n(T)f=n^{-r}F_n$ are unbounded in $\x_r$. Using the Leibniz rule and the triangle inequality we obtain  $n>r+1$
$$n^{-r}|F_n^{(r+1)}(z)|\ge \frac{a_r}{|1-z|} - b_r\sum_{j=1}^{r+1} \frac{n^{-j}}{|1-z|^{j+1}}\,,$$
with $a_r,b_r>0$ independent of $n$ and $z$. The same standard estimates for integrals mentioned above yield for $\rho_n=1-\frac1{n}$
$$n^{-r}\int_0^{2\pi}|F_n^{(r+1)}(\rho_n e^{it})|dt\ge a'_r\log n -b_r'\,,$$
with $a'_r,b'_r>0$ independent of $n$, and the claim together with (ii) follow.\\
 (iii) For $n>r+1$ and  $f(z)=1$ we have  $$n^{r+1}\lesssim \|T^nf\|_r=n(n-1)\ldots (n-r)\,, $$
which gives the desired conclusion.
\end{proof}

In particular, without additional assumptions on the peripheral spectrum the conclusion of Corollary \ref{nevanlinna3} may fail.
\medskip

{\bf Acknowledgements.} The authors are grateful to
the referee for his suggestions and useful comments which improve the
original version.

The second named author would like to express his gratitude to the
University of Lund, Centre for Mathematical Sciences, Sweden
for hospitality and excellent working conditions. He was also supported by a Project financed from Lucian Blaga University of Sibiu research grants LBUS-IRG-2015-01.

\medskip

{\bf References}
\medskip

[AS] J. Agler and M. Stankus, {\it $m$-Isometric transformations of Hilbert space, I}, Integr. Equat. Oper. Th. Vol. 21 (1995), 383-429.

[AP] A. Aleman and A-M. Persson, {\it Resolvent estimates and decomposable extensions of generalized Ces\`aro operators}, J. Funct. Anal. 258 (2010), no. 1, 67-98.

[Al] G. R. Allan, {\it Power-bounded elements and radical Banach
algebras}, Banach Center Publ. 38, Institute of Mathematics, Polish
Academy of Sciences, Warszawa, 1997, 9-16.

[AB] S. I. Ansari and P. S. Bourdon, {\it Some properties of cyclic
operators}, Acta Sci. Math. (Szeged) 63 (1997), 195-207.




[BC] F. Bayart and G. Costakis, {\it Cyclic operators with finite
support}, Israel Journal of Mathematics, (2012), 1-37.

[BM] F. Bayart and E. Matheron, {\it Dynamics of linear operators},
Cambridge University Press, 2009.

[BG] E. Berkson and T. A. Gillespie {\it Spectral decompositions, ergodic averages, and the Hilbert transform}, Studia Math., 144 (1), 2001, 39-61.



[B] J. Boos, {\it Classical and modern methods in summability} Assisted by Peter Cass. Oxford Mathematical Monographs. Oxford Science Publications. Oxford University Press, Oxford, 2000.

[C] L. W. Cohen, {\it On the Mean Ergodic Theorem}, Ann. of Math. Second Series, Vol. 41, No. 3, (1940), 505-509.


[Du] P.L. Duren,  {\it Theory of $H^p$ spaces}, Pure and Applied Mathematics, Vol. 38 Academic Press, New York-London 1970

[DyS] K. Dykema and H. Schultz, {\it Brown measure and iterates of the Aluthge transform for some operators arising from measurable actions}, Trans. Amer. Math. Soc. 361 (2009), 6583-6593.

[Eb] W. F. Eberlein, {\it Abstract Ergodic Theorems and Weak Almost Periodic Functions}, Trans. Amer. Math. Soc. 67 (1949), 217-240.

[Ed] E. Ed-Dari, {\it On the $(C,\alpha)$ uniform ergodic theorem}, Studia Math. 156 (1), 2003, 3-13.



[FR] O. El-Fallah and T. Ransford, {\it Extremal growth of powers of
operators satisfying resolvent conditions of Kreiss-Ritt type}, J.
Funct. Anal. 196 (2002), 135-154.


[GLM] M. Gonzalez, F. Le\'on-Saavedra, A. Montes-Rodr\'iguez, {\it Semi-Fredholm theory: hypercyclic and supercyclic subspaces}, Proc. London Math Soc. (3) 81 (2000), no. 1, 169-189.

[GH] J. J. Grobler and C. B. Huijsmans, {\it Doubly Abel bounded operators with single spectrum}, Questiones Mathematicae, 18 (1995), 397-406.

[He] D. A. Herrero, {\it Limits of hypercyclic and supercyclic
operators}, J. Funct. Anal., 99 (1991), 179-190.

[Hi] E. Hille, {\it Remarks on ergodic theorems}, Trans. Amer. Math.
Soc. 57 (1945), 246-269.


[K\'e] L. K\'erchy, {\it Operators with regular norm-sequences},
Acta Sci. Math. (Szeged) 63 (1997), 571-605.

[Kr] U. Krengel, {\it Ergodic Theorems}, Walter de Gruyter Studies
in Mathematics 6, Berlin - New York, 1985.

[Ku] M. K. Kuo, {\it Tauberian conditions for almost convergence}, Positivity 13 (2009), 611-619.

[L] G. G. Lorentz, {\it A contribution to the theory of divergent sequences}, Acta Math., 80
(1948), 167-190.

[LSS], M. Lin, D. Shoikhet and L. Suciu, {\it Remarks on uniform ergodic theorems}, Acta Sci Math. (Szeged), 81 (2015), 251-283.


[MSZ] A. Montes-Rodr\'iguez, J. S\'anchez-\'Alvarez, and J.
Zem\'anek, {\it Uniform Abel-Kreiss boundedness and the extremal
behaviour of the Volterra operator}, Proc. London Math. Soc. 91
(2005), 761-788.

[N] O. Nevanlinna, {\it Resolvent conditions and powers of operators}, Studia Math. 145 (2) (2001), 113-134.

[R] A. V. Romanov, {\it Weak convergence of operator means}
Izvestiya: Mathematics 75:6, 2011, 1165-1183.

[Sc] M. Schreiber, {\it Uniform families of ergodic operator nets}, Semigroup Forum, Vol. 86, Issue 2, (2013), 321-336.

[Sh] A. L. Shields, {\it On M\"{o}bius bounded operators}, Acta
Sci. Math. (Szeged) 40 (1978), 371-374.

[SW] J. C. Strikwerda and B. A. Wade, {\it A survey of the Kreiss
matrix theorem for power bounded families of matrices and its
extensions,} Banach Center Publ. 38, Institute of Mathematics,
Polish Academy of Sciences, Warszawa, 1994, 339-360.

[SZ], L. Suciu and J. Zem\'anek, {\it Growth conditions and Ces\`aro
means of higher order}, Acta Sci Math. (Szeged), 79 (2013), 545-581.

[V] Q. P. V\~u, {\it A short proof of the Y. Katznelson's and L. Tzafriri's theorem},
Proc. Amer. Math. Soc., 115 (1992), 1023-1024.

\end{document}